\documentclass[12pt]{amsart}
\usepackage{amssymb}
\usepackage[all]{xy}

\swapnumbers %to put numbers in the front of proclamations
\newtheorem*{thm*}{Theorem}
\newtheorem*{lem*}{Lemma}
\newtheorem*{prop*}{Proposition}
\newtheorem*{cor*}{Corollary}
\theoremstyle{definition}

\newtheorem*{example*}{Example}

\theoremstyle{remark}

\newtheorem*{rmk*}{Remark}
\numberwithin{equation}{subsection}

\newcommand{\divided}[2]{#1^{(#2)}}

\newcommand{\UU}{\mathbf{U}}
\newcommand{\Sch}{\mathbf{S}}

\newcommand{\bil}[2]{\langle #1, #2 \rangle}

\renewcommand{\ge}{\geqslant}

\newcommand{\phat}{\widehat{p}}

\begin{document}
\title[Quantized enveloping algebras via inverse limits] {Addendum to:\\
  ``Constructing quantized enveloping algebras via inverse limits of
  finite dimensional algebras''}

% author 1 information
\author[S.~Doty]{Stephen Doty}
\address{Department of Mathematics and Statistics\\ 
Loyola University Chi\-cago\\ 
Chicago, Illinois 60626 USA}
%\curraddr{}
\thanks{Supported by a Mercator grant from the DFG}
  \email{doty@math.luc.edu}

\date{17 October 2009}
% Use this \subjclass if you are using amsart version 2.0 (December 1999).
%\subjclass[2000]{17B37, 16W35, 81R50}
% Use this one if you are using an older version of amsart.
%\subjclass{17B37, 16W35, 81R50}
%\keywords{q-Schur algebras, generalized Schur algebras, quantized enveloping algebras}

\begin{abstract} 
It is shown that the question raised in Section 5.7 of \cite{IL} has
an affirmative answer.
\end{abstract}
\maketitle

\parskip=2pt 
\allowdisplaybreaks 

%\section*{Introduction}\noindent 

\noindent
We use the notation and numbering from \cite{IL}. In Section 5.7 we
raised the question: does ${}_R \UU$ embed in ${}_R \widehat{\UU}$?
The purpose of this addendum is to show that the answer is
affirmative. To be precise, we have the following result.

\begin{thm*}
  The map ${}_R\theta \colon {}_R \UU \to {}_R \widehat{\UU}$ defined
  in 5.7(a) is injective. Hence, ${}_R \UU$ is isomorphic
  to the $R$-subalgebra of ${}_R \widehat{\UU}$ generated by all
  $\widehat{E}_{\pm i}^{(m)}$ ($i \in I$, $m \ge 0$) and
  $\widehat{K}_h$ ($h \in Y$).
\end{thm*}

\begin{proof}
Consider the commutative diagram of $R$-algebra maps
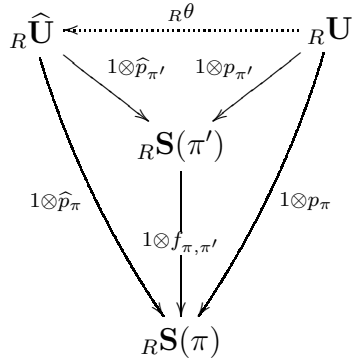
\[
\xymatrix{
_R\widehat{\UU} \ar[dr]^{1\otimes\phat_{\pi'}} 
\ar@/_/[dddr]_{1\otimes\phat_\pi} & & 
_R\UU \ar[dl]_{1\otimes p_{\pi'}} 
\ar@/^/[dddl]^{1\otimes p_\pi}
\ar@{.>}[ll]_{_R\theta} \\ 
& _R\Sch(\pi') \ar[dd]|{1\otimes f_{\pi,\pi'}} \\ \\
& _R\Sch(\pi) 
}
\]
for any finite saturated $\pi \subset \pi'$.  The universal property
of inverse limits guarantees the existence of a unique $R$-algebra map
$_R\theta \colon {}_R \UU \to {}_R\widehat{\UU}$ making the diagram
commute, and one easily checks that this map coincides with the map
defined in 5.7(a). We need to show that ${}_R \theta$ is injective.

We note that from the definitions it follows that for any $\pi$ the
maps $p_\pi$ and $\dot{p}_\pi$ are related by the identity
$1_\lambda\, p_\pi(u)\, 1_\mu = \dot{p}_\pi( 1_\lambda\, u \, 1_\mu)$,
for any $u \in \UU$, $\lambda, \mu \in X$.  It follows immediately
that
\[
1_\lambda\, (1\otimes p_\pi)(u)\, 1_\mu = (1 \otimes \dot{p}_\pi)(
1_\lambda\, u \, 1_\mu),
\]
for any $u \in {}_R\UU$, $\lambda, \mu \in X$. This is needed below.

Let $u \in \ker {}_R \theta$ and $\lambda, \mu \in X$. Then
$\widehat{1}_\lambda \, {}_R \theta(u) \, \widehat{1}_\mu = 0$ in
${}_R\widehat{\UU}$. This implies that ${1}_\lambda \, (1\otimes
p_\pi)(u) \, {1}_\mu = 0$ in ${}_R \Sch(\pi)$ for any $\pi$, and hence
that $(1 \otimes \dot{p}_\pi)(1_\lambda\, u \, 1_\mu) = 0$ in ${}_R
\Sch(\pi)$ for any $\pi$. Thus by Lemma 5.2 we have ${1}_\lambda \, u
\, {1}_\mu \in \cap_\pi\, {}_R \dot{\UU}[\pi^c]$. Since the
intersection is zero, the equality ${1}_\lambda \, u \, {1}_\mu = 0$
holds in ${}_R \dot{\UU}$, for any $\lambda, \mu \in X$.  We claim
this implies that $u = 0$.

To see the claim we observe that the construction of $\dot{\UU}$ given
in Section 3.1 and \cite[Chapter 23]{Lusztig} commutes with change of
scalars. This is easily verified and left to the reader. It means that
${}_R\pi_{\lambda, \mu}(u) = 0$ where
\[
{}_R\pi_{\lambda, \mu} \colon {}_R \UU \to {}_R \UU/ \Big( \sum_{h \in
  Y} (K_h- \xi^{\bil{h}{\lambda}}){}_R\UU + \sum_{h \in Y}
{}_R\UU(K_h- \xi^{\bil{h}{\mu}}) \Big)
\]
is the canonical projection map. Thus it follows that $$u \in \sum_{h
  \in Y} (K_h- \xi^{\bil{h}{\lambda}}){}_R\UU + \sum_{h \in Y}
{}_R\UU(K_h- \xi^{\bil{h}{\mu}}).$$ Since this is true for all
$\lambda, \mu \in X$ it follows that $u = 0$ as claimed.
\end{proof}

From \cite[31.1.5]{Lusztig} we recall the category ${}_R\mathcal{C}$
of unital ${}_R \dot{\UU}$-modules. As in \cite[23.1.4]{Lusztig} one
easily checks that this is the same as the category of ${}_R
\UU$-modules admitting a weight space decomposition.  Following
\cite[31.2.4]{Lusztig}, we say that an object $M$ of ${}_R\mathcal{C}$
is \emph{integrable} if for any $m \in M$ there exists some $n_0$ such
that $$\divided{E_{i}}{n} \, m = 0 = \divided{E_{-i}}{n} \, m$$ for
all $n \ge n_0$.

We have the following consequence of the theorem, which generalizes
\cite[Proposition 5.11]{Jantzen:LQG} and \cite[Proposition
  3.5.4]{Lusztig}.

\begin{cor*}
  Suppose that $u \in {}_R \UU$ acts as zero on all integrable objects
  of ${}_R \mathcal{C}$. Then $u = 0$.
\end{cor*}

\begin{proof}
The natural quotient map $1 \otimes p_\pi \colon {}_R \UU \to {}_R
\Sch(\pi)$ makes ${}_R \Sch(\pi)$ into a left ${}_R \UU$-module, by
defining $u \cdot s = \overline{u}\,s$ (for $u \in {}_R\UU$, $s\in
{}_R\Sch(\pi)$) where $\overline{u}$ is the image of $u$.  It is
easily checked that, as a left ${}_R \UU$-module, ${}_R\Sch(\pi)$ is
an integrable object of ${}_R\mathcal{C}$. Hence by hypothesis $u$
acts as zero on ${}_R\Sch(\pi)$, for any finite saturated subset $\pi$
of $X^+$. It follows that $u$ lies in the intersection of the kernels
of the various $1 \otimes p_\pi$. By the commutative diagram above
this implies that ${}_R \theta(u) = 0$. By the theorem, $u = 0$. 
\end{proof}

\end{document}